\documentclass[hidelinks]{amsart}

\usepackage[defaultlines=2,all]{nowidow} 

\usepackage{ mathtools, amssymb, stmaryrd, mathdots, bbm, mleftright, faktor }

\usepackage{ letltxmacro }    
\usepackage{ xparse }
\usepackage{ stackengine, graphicx, scalerel, etoolbox }
\usepackage{xpatch} 

\usepackage[shortlabels]{ enumitem }

\usepackage{ amsthm, thmtools, url }
\usepackage[linktoc=all,pdftitle={Spin characters of \^{A}n proportional to linear characters}]{ hyperref }
\usepackage[nameinlink, capitalize]{ cleveref }

\crefname{definition}{Def.}{Defs.}
\Crefname{definition}{Definition}{Definitions}

\numberwithin{equation}{section}
\declaretheorem[style=plain,numberlike=equation]{theorem}
\declaretheorem[style=plain,numberlike=theorem]{lemma}

\declaretheorem[style=remark,numberlike=theorem]{remark}

\makeatletter    
\newcommand{\prelistcommand}{\nobreak\leavevmode\@nobreaktrue}
\makeatother

\SetEnumitemKey{beginthm}{before=\prelistcommand}   
\SetEnumitemKey{smallin}{leftmargin=2.7em, labelindent=*}   

\setlist{smallin, topsep=1pt}

\crefformat{section}{\S#2#1#3}
\crefrangeformat{section}{\S#3#1#4--#5#2#6}
\crefmultiformat{section}{\S#2#1#3}{ and~\S#2#1#3}{, \S#2#1#3}{, and~\S#2#1#3}
\crefrangemultiformat{section}{\S#2#1#3}{ and~\S#2#1#3}{, \S#2#1#3}{, and~\S#2#1#3}

\crefformat{equation}{(#2#1#3)}
\crefrangeformat{equation}{(#3#1#4)--(#5#2#6)}
\crefmultiformat{equation}{(#2#1#3)}{ and~(#2#1#3)}{, (#2#1#3)}{, and~(#2#1#3)}
\crefrangemultiformat{equation}{(#2#1#3)}{ and~(#2#1#3)}{, (#2#1#3)}{, and~(#2#1#3)}

\crefname{enumi}{}{}

\makeatletter
\newcommand{\proofsubheading}[1]{%
  \par
  \addvspace{\medskipamount}%
  \noindent\emph{#1}\par\nobreak
  \addvspace{\smallskipamount}%
  \@afterheading
}
\makeatother

\makeatletter
\let\cite\relax
\DeclareRobustCommand{\cite}{%
  \let\new@cite@pre\@gobble
  \@ifnextchar[\new@cite{\@citex[]}}
\def\new@cite[#1]{\@ifnextchar[{\new@citea{#1}}{\@citex[#1]}}
\def\new@citea#1{\def\new@cite@pre{#1}\@citex}
\def\@cite#1#2{[{\new@cite@pre\space#1\if\relax\detokenize{#2}\relax\else, #2\fi}]}
\makeatother

\newcommand{\initialsspace}{0.1em}
 
\makeatletter
\newcommand{\splitlist}[1]{\@splitlist#1\@nil}
\def\@splitlist#1\@nil{%
  \if\relax\detokenize{#1}\relax
    \expandafter\@gobble
  \else
    \expandafter\@firstofone
  \fi
  {\@spl@tlist#1.\@nil}%
}
 
\def\@spl@tlist#1.#2\@nil{%
    \def\tmpA{#1}%
    \def\tmpB{#2}%
    \def\tmpP{.}%
    \ifx\tmpB\tmpP%
        #1.%
    \else{%
        \ifx\tmpA\@empty%
        \else%
                #1.\nobreak\hspace{\initialsspace}%
        \fi%
    }%
    \fi%
  \if\relax\detokenize{#2}\relax
    \expandafter\@firstoftwo
  \else
    \expandafter\@secondoftwo
  \fi
  {\unskip}%
  {\@spl@tlist#2\@nil}%
}
\makeatother

\NewDocumentCommand\set{s m}{%
    \IfBooleanTF#1%
    {\left\{ #2 \right\}}%
    {\{#2\}}%
}
\NewDocumentCommand\setbuild{s m m}{%
    \IfBooleanTF#1%
    {\ensuremath{\left\{\, #2 \, \middle| \, #3 \,\right\}}}%
    {\ensuremath{\{\, #2 \, \mid \, #3 \,\}}}%
}
\NewDocumentCommand\spangle{s m m m}{%
    \IfBooleanTF#1%
    {\ensuremath{\left\langle\, #2 \, \middle| \, #3 \,\right\rangle_{#4}}}%
    {\ensuremath{\langle\, #2 \, \mid \, #3 \,\rangle_{#4}}}%
}

\makeatletter
\newcommand\notni{\mathrel{\m@th\mathpalette\canc@l\owns}}
\newcommand\canc@l[2]{{\ooalign{$\hfil#1/\mkern1mu\hfil$\crcr$#1#2$}}}
\makeatother

\DeclarePairedDelimiter{\ceil}{\lceil}{\rceil}
\DeclarePairedDelimiter{\floor}{\lfloor}{\rfloor}

\DeclarePairedDelimiter{\abs}{\lvert}{\rvert}

\newcommand{\R}{\mathbb{R}}

\newcommand{\eps}{\epsilon}
\newcommand{\la}{\lambda}

\newcommand{\downto}[1]{\mathord{\downarrow}_{#1}}

\DeclareMathOperator{\diag}{diag}

\makeatletter
\NewCommandCopy\@@pmod\pmod
\DeclareRobustCommand{\pmod}{\@ifstar\@pmods\@@pmod}
\def\@pmods#1{\mkern4mu({\operator@font mod}\mkern 6mu#1)}
\makeatother

\renewcommand{\epsilon}{\varepsilon}
\renewcommand{\phi}{\varphi}

\renewcommand{\geq}{\geqslant}

\newcommand\ppmod[1]{\ (\operatorname{mod}\,#1)}

\makeatletter
\let\save@mathaccent\mathaccent
\newcommand*\if@single[3]{%
  \setbox0\hbox{${\mathaccent"0362{#1}}^H$}%
  \setbox2\hbox{${\mathaccent"0362{\kern0pt#1}}^H$}%
  \ifdim\ht0=\ht2 #3\else #2\fi
  }
\newcommand*\rel@kern[1]{\kern#1\dimexpr\macc@kerna}
\DeclareRobustCommand*\widebar[1]{\@ifnextchar^{{\wide@bar{#1}{0}}}{\wide@bar{#1}{1}}}
\newcommand*\wide@bar[2]{\if@single{#1}{\wide@bar@{#1}{#2}{1}}{\wide@bar@{#1}{#2}{2}}}
\newcommand*\wide@bar@[3]{%
  \begingroup
  \def\mathaccent##1##2{%
    \let\mathaccent\save@mathaccent
    \if#32 \let\macc@nucleus\first@char \fi
    \setbox\z@\hbox{$\macc@style{\macc@nucleus}_{}$}%
    \setbox\tw@\hbox{$\macc@style{\macc@nucleus}{}_{}$}%
    \dimen@\wd\tw@
    \advance\dimen@-\wd\z@
    \divide\dimen@ 3
    \@tempdima\wd\tw@
    \advance\@tempdima-\scriptspace
    \divide\@tempdima 10
    \advance\dimen@-\@tempdima
    \ifdim\dimen@>\z@ \dimen@0pt\fi
    \rel@kern{0.6}\kern-\dimen@
    \if#31
      \overline{\rel@kern{-0.6}\kern\dimen@\macc@nucleus\rel@kern{0.4}\kern\dimen@}%
      \advance\dimen@0.4\dimexpr\macc@kerna
      \let\final@kern#2%
      \ifdim\dimen@<\z@ \let\final@kern1\fi
      \if\final@kern1 \kern-\dimen@\fi
    \else
      \overline{\rel@kern{-0.6}\kern\dimen@#1}%
    \fi
  }%
  \macc@depth\@ne
  \let\math@bgroup\@empty \let\math@egroup\macc@set@skewchar
  \mathsurround\z@ \frozen@everymath{\mathgroup\macc@group\relax}%
  \macc@set@skewchar\relax
  \if#31
    \macc@nested@a\relax111{#1}%
  \else
    \def\gobble@till@marker##1\endmarker{}%
    \futurelet\first@char\gobble@till@marker#1\endmarker
    \ifcat\noexpand\first@char A\else
      \def\first@char{}%
    \fi
    \macc@nested@a\relax111{\first@char}%
  \fi
  \endgroup
}
\makeatother

\usepackage{letltxmacro}
\makeatletter
\let\oldr@@t\r@@t
\def\r@@t#1#2{%
\setbox0=\hbox{$\oldr@@t#1{#2\,}$}\dimen0=\ht0
\advance\dimen0-0.2\ht0
\setbox2=\hbox{\vrule height\ht0 depth -\dimen0}%
{\box0\lower0.4pt\box2}}
\LetLtxMacro{\oldsqrt}{\sqrt}
\renewcommand*{\sqrt}[2][\ ]{\oldsqrt[#1]{#2}}
\makeatother

\newcommand{\al}{\alpha}

\newcommand{\ka}{\kappa}

\newcommand{\lap}[1]{\la^{(#1)}}

\newcommand\hsss{\hat{S}_}
\newcommand\haaa{\hat{A}_}

\newcommand{\ass}[1]{#1^{\mathsf{a}}}
\newcommand{\conj}[1]{#1^{\mathsf{c}}}

\newcommand{\symchar}[1]{\chi^{#1}}
\newcommand{\altchar}[1]{\phi^{#1}}
\newcommand{\altcharnsc}[1]{\phi^{#1}} 

\newcommand{\spsymchar}[1]{\langle #1 \rangle}
\newcommand{\spaltcharsc}[1]{[#1]} 
\newcommand{\spaltchar}[1]{[#1]}

\newcommand\br[1]{\widebar{#1}} 

\newcommand{\corandquot}[3]{[#1; (#2,#3)]}
\newcommand{\twoc}[1]{\ka_{#1}}
\newcommand\fcor[1]{\bar{\ka}_{#1}}
\newcommand{\btwoc}[1]{\fcor{#1}}
\newcommand{\dbl}[1]{\mathrm{dbl}(#1)} 

\newcommand{\notsc}[1]{\eta(#1)}

\calclayout

\usepackage{amsaddr}

\title[Spin characters of \(\haaa n\) proportional to linear characters]{Spin characters of the alternating group \\ which are proportional to linear characters \\ in characteristic \(2\)}
\author{\vspace{-0.4cm} Eoghan McDowell}
\address{\upshape \vspace{-0.1cm}%
University of Bristol \endgraf
\texttt{eoghan.mcdowell@bristol.ac.uk} \endgraf
\href{https://eoghanjmcdowell.com}{\texttt{eoghanjmcdowell.com}}
}

\begin{document}

\begin{abstract}
We classify when a spin and a non-spin irreducible character of the alternating group have proportional \(2\)-modular reductions.
Equivalently, we classify when such a pair of characters are proportional on elements of odd order.
\end{abstract}

\vspace*{-0.1cm}
\maketitle

\thispagestyle{empty}

\vspace*{-0.6cm}
\section{Introduction}

The \emph{double covers} of the symmetric and alternating groups,
being their Schur covers, 
control the projective character theory of the original groups. 
Non-linear projective characters correspond to characters of the double cover which are not lifted from the original group, known as \emph{spin characters}.
The central involution of a double cover acts on an irreducible representation by multiplication by \(\pm1\); the spin characters are precisely those afforded by representations where the central involution acts by \(-1\).
Over a field of characteristic \(2\), this action becomes trivial -- so reduction modulo \(2\) converts a spin character to a linear (i.e.~non-spin) character of the original group.

This paper classifies, for the double cover \(\haaa n\) of the alternating group,
when the \(2\)-modular reduction of an irreducible spin character is the same as the \(2\)-modular reduction of an irreducible non-spin character.
Here, ``same'' means having the same multiset of composition factors.
More generally, we classify when such characters have proportional \(2\)-modular reductions, in the sense of having proportional composition factor multiplicities.
This problem was solved for the double cover \(\hsss n\) of the symmetric group by Fayers and the present author \cite{fayersmcd2025spintospecht}; we deduce our main theorem from that result using the Clifford theory of the symmetric and alternating groups.

An equivalent description of our result is a classification of proportional rows in the \(2\)-modular decomposition matrix of the double cover of the alternating group.
Another interpretation is through \emph{Brauer characters}: restricting an ordinary character to the \(p'\)-classes (i.e.~conjugacy classes of elements of order coprime to \(p\)) gives the Brauer character of the \(p\)-modular reduction.
Thus our result equivalently classifies ordinary characters which are proportional on \(2'\)-classes.

Our main result is as follows.
We write \(\br{\chi}\) for the Brauer character of the \(2\)-modular reduction of an ordinary character \(\chi\).
If \(\chi\) is a character of \(\haaa n\), we write \(\conj{\chi}\) for the \(\hsss n\)-conjugate character, obtained by precomposing with conjugation by an element of \(\hsss n \setminus \haaa n\).
The labelling of the characters of \(\haaa n\) by partitions and the notation for partitions used below is explained in \Cref{sec:labelling}, and the result is proved in \Cref{sec:proof}.

\begin{theorem}
\label{thm:main}
Let \(\phi\) be an irreducible non-spin character of \(\haaa n\) labelled by a partition \(\la\), and let \(\psi\) be an irreducible spin character of \(\haaa n\) labelled by a strict partition \(\al\).
Then \(\br{\phi}\) is proportional to one (or both) of \(\br{\psi}\) and \(\br{\conj{\psi}}\)
if and only if there exist nonnegative integers \(a\), \(r\) and \(s\) such that:
\begin{itemize}
    \item \(\la = \corandquot{\twoc{a}}{\twoc{r}}{\twoc{s}}\), and
    \item \(\al = \btwoc{a} \sqcup 2(\twoc{r}+\twoc{s})\), and
    \item
    either:
    \begin{itemize}[label=\(\triangleright\),leftmargin=1.7em]
        \item 
        \(r \neq s\), in which case \(\phi = \conj{\phi}\) and \(\br{\psi} = \br{\conj{\psi}} = 2^{\floor{\max\{r,s\}/2}+1-\eps(\al)} \br{\phi}\); or
        \item
        \(r=s=0\), in which case \(\{\br{\phi}, \br{\conj{\phi}}\} = \{\br{\psi}, \br{\conj{\psi}}\}\).
    \end{itemize}
\end{itemize}
\end{theorem}

For a discussion on the partitions appearing in the theorem, see \cite[\S 1.2]{fayersmcd2025spintospecht}, where strict partitions of this form are called \emph{\(4\)-stepped-and-semicongruent}.

For both the symmetric and alternating groups, it is only for \(p=2\) that it is possible for a spin and non-spin character to have the same \(p\)-modular reduction.
Meanwhile, for a pair of spin characters or a pair of non-spin characters, equal and proportional \(p\)-modular reductions were classified by Wildon \cite{wildon2008distinctrows}, the author \cite{mcdowell2024charsonlprimeclasses}, and Fayers and the author \cite{fayersmcd2026proportionalspinpairs}, with the exception of pairs of non-spin characters of the alternating group when \(p=3\).
In most cases, only associate and conjugate pairs have equal \(p\)-modular reductions; the exceptions are the spin/non-spin pairs when \(p=2\) (classified in \cite{fayersmcd2025spintospecht} and the present paper), pairs of spin characters when \(p=3\) (classified in \cite{fayersmcd2026proportionalspinpairs}), and pairs of non-spin characters of the alternating group when \(p=3\) (which remain unclassified).

\section{Characters of \texorpdfstring{\(\hsss n\)}{Ŝn} and \texorpdfstring{\(\haaa n\)}{Ân} and their labelling}
\label{sec:labelling}

The double covers \(\hsss n\) and \(\haaa n\) are constructed in \cite[pp.~18--23]{hoffmanhumphreys1992projreps} (there are two double covers of the symmetric group, but they are \emph{isoclinic}, so our results apply to either choice \cite[\S6.7]{ATLAS}).
The theory of their non-spin characters is essentially the (linear) character theory of the symmetric and alternating groups, an account of which can be found in \cite{jameskerber1984reptheory}.
For an account of the spin character theory, see \cite{hoffmanhumphreys1992projreps}.
Here we summarise the results we need.


\subsection{Partitions}

A \emph{partition} of an integer \(n\) is a weakly decreasing sequence of nonnegative integers whose sum is \(n\).
A \emph{strict partition} is a partition all of whose parts are distinct.

Given two partitions \(\la\) and \(\mu\), we write \(\la \sqcup \mu\) for the partition obtained by concatenating the parts of \(\la\) and \(\mu\) and rearranging into weakly decreasing order, and we write \(\la + \mu\) for the partition with parts \((\la+\mu)_i = \la_i + \mu_i\).

The \emph{conjugate} of a partition \(\la\), denoted \(\la'\), is the partition with parts
\[
    \la'_i = \abs{\setbuild{j}{\la_j \geq i}}
\]
(or, equivalently, \(\la'\) is the partition obtained by drawing \(\la\) as a \emph{Young diagram} and reflecting in the main diagonal, exchanging rows and columns).
Write
\[
    \notsc{\la} = \begin{cases}
        1 & \text{ if \(\la \neq \la'\);}\\
        0 & \text{ if \(\la = \la'\).} \\
    \end{cases}
\]

We say a strict partition \(\al\) is \emph{odd} if \(\al\) has an odd number of even parts (i.e.~if a permutation of cycle type \(\al\) is odd) and \emph{even} otherwise.
Write
\[
    \eps(\al) = \begin{cases}
        1 & \text{ if \(\al\) is odd;}\\
        0 & \text{ if \(\al\) is even.}\\
    \end{cases}
\]

It is convenient for us to describe a partition in terms of its \emph{\(2\)-core} and \emph{\(2\)-quotient}.
For \(\la\) a partition and \(p\) a positive integer, the \emph{\(p\)-core} of \(\la\) is the partition obtained from \(\la\) by successively removing \emph{\(p\)-hooks}, and the \(p\)-quotient of \(\la\) is a \(p\)-tuple of partitions encoding the removed \(p\)-hooks (see \cite[pp.~55--56 and \S2.7]{jameskerber1984reptheory} for definitions).
For our purposes, it suffices to know that a partition is uniquely determined by its \(p\)-core and \(p\)-quotient, and that a \(2\)-core is of the form \((r, r-1, r-2, \ldots, 1)\) for some \(r \geq 0\).
We write \(\twoc{r} = (r, r-1, r-2, \ldots, 1)\) for the \(2\)-core of length \(r\), and we write
\(\corandquot{\twoc{r}}{\lap0}{\lap1}\) for the partition with \(2\)-core \(\twoc{r}\) and \(2\)-quotient \((\lap0, \lap1)\).

We write \(\btwoc{r} = (2r-1, 2r-5, 2r-9, \ldots, \bar{r})\) for \(r \geq 0\), where \(\bar{r} \in \set{1,3}\) is the residue of \(2r-1 \ppmod{4}\).
The strict partitions \(\btwoc{r}\) for \(r \geq 0\) are the possible \emph{\(4\)-bar-cores} (where \emph{bar-cores} are the analogue of cores for strict partitions), and satisfy \(\dbl{\btwoc{r}} = \twoc{r}\) (where the \emph{double} \(\dbl{\al}\) of a strict partition \(\al\) is obtained by replacing every part \(x\) with two parts \(\ceil{x/2}\) and \(\floor{x/2}\)), though here we require only the notation.

\subsection{Non-spin characters}

The irreducible non-spin characters of \(\hsss n\) are bijectively labelled by the partitions of \(n\). Write \(\chi^\la\) for the irreducible non-spin character of \(\hsss n\) labelled by a partition \(\la\).

The restriction of \(\chi^\la\) to \(\haaa n\) is either irreducible and self-\(\hsss n\)-conjugate (if \(\la \neq \la'\)), or is the sum of a \(\hsss n\)-conjugate pair of irreducible characters (if \(\la = \la'\)).
All irreducible non-spin characters of \(\haaa n\) arise this way.
Write \(\altchar{\la}\) for one of the irreducible constituents of \(\symchar{\la}\downto{\haaa n}\), so that
\[
    \symchar{\la}\downto{\haaa n} = \symchar{\la'}\downto{\haaa n} = \begin{cases}
        \altcharnsc{\la} & \text{ if \(\la\neq\la'\);}\\
        \altchar{\la} + \conj{(\altchar{\la})} & \text{ if \(\la=\la'\).} \\
    \end{cases}
\]
If \(\la=\la'\), the conjugate pair \(\altchar{\la}\) and \(\conj{(\altchar{\la})}\) agree on all classes except those of elements projecting to cycle type \(\diag(\la)\), where \(\diag(\la)\) is the partition whose parts are the diagonal hook lengths of \(\la\), and these are the only classes on which the characters take irrational values \cite[Theorem~2.5.13]{jameskerber1984reptheory}.

\subsection{Spin characters}

The \emph{associate} of a character \(\chi\) of \(\hsss n\), denoted \(\ass{\chi}\), is the character obtained from \(\chi\) by tensoring with the sign character.

Each irreducible spin characters of \(\hsss n\) or its associate is labelled by a strict partition of \(n\): for a strict partition \(\al\), there is an irreducible spin character labelled by \(\al\) which we denote \(\spsymchar{\al}\), and
\[
    \setbuild{ \spsymchar{\al} }{\text{\(\al\) an even strict partition of \(n\)}}
    \,\sqcup\,
    \setbuild{\spsymchar{\al},\, \ass{\spsymchar{\al}}}{\text{\(\al\) an odd strict partition of \(n\)}}
\]
is a complete irredundant list of the irreducible spin characters of \(\hsss n\).

As in the case of non-spin characters, the restriction of \(\spsymchar{\al}\) to \(\haaa n\) is either irreducible and self-\(\hsss n\)-conjugate (if \(\al\) is odd), or is the sum of a \(\hsss n\)-conjugate pair of irreducible characters (if \(\al\) is even), and all irreducible spin characters of \(\haaa n\) arise this way.
Write \(\spaltchar{\al}\) for one of the irreducible constituents of \(\spsymchar{\al}\downto{\haaa n}\), so that
\[
    \spsymchar{\al}\downto{\haaa n} = \ass{\spsymchar{\al}}\downto{\haaa n} = \begin{cases}
        \spaltchar{\al} & \text{ if \(\al\) odd;}\\
        \spaltchar{\al} + \conj{\spaltchar{\al}} & \text{ if \(\al\) even.} \\
    \end{cases}
\]
If \(\al\) is even, the conjugate pair \(\spaltchar{\al}\) and \(\conj{\spaltchar{\al}}\) agree on all classes except those of elements projecting to cycle type \(\al\), and these are the only classes on which the characters take irrational values \cite[Theorem~8.7(iv)]{hoffmanhumphreys1992projreps}.

\section{Proof}
\label{sec:proof}

\begin{lemma}
\label{lemma:iffrestrictiontoAn}
Let \(\chi_1\) and \(\chi_2\) be characters of \(\hsss n\) and $m\in\R$. Then \(\br{\chi_1} = m \br{\chi_2}\) if and only if \(\br{\chi_1\downto{\haaa n}} = m \br{\chi_2\downto{\haaa n}}\).
\end{lemma}

\begin{proof}
Elements of \(\hsss n \setminus \haaa n\) are of even order, so character values on these elements are irrelevant to the \(2\)-modular reductions.
\end{proof}

\begin{lemma}
\label{lemma:Anproportionality_implies_Snproportionality}
Let \(\la\) be a partition of \(n\), let \(\al\) a strict partition of \(n\), and let \(m \in \R\).
If \(\br{\altchar{\la}} = m \br{\spaltchar{\al}}\), then \(\br{\symchar{\la}} = 2^{\eps(\al)-\notsc{\la}} m \br{\spsymchar{\al}}\).
\end{lemma}

\begin{proof}
Suppose \(\br{\altchar{\la}} = m \br{\spaltchar{\al}}\).
Then \(\br{\conj{(\altchar{\la})}} = m \br{\conj{\spaltchar{\al}}}\), and so \(\br{\altchar{\la} + \conj{(\altchar{\la})}} = m (\br{\spaltchar{\al} + \conj{\spaltchar{\al}}})\).
But \(\altchar{\la} + \conj{(\altchar{\la})} = 2^{\notsc{\la}} \symchar{\la}\downto{\haaa n}\) and \(\spaltchar{\al} + \conj{\spaltchar{\al}} = 2^{\eps(\al)} \spsymchar{\al}\downto{\haaa n}\), and the result follows by \Cref{lemma:iffrestrictiontoAn}.
\end{proof}

\begin{remark}
The converse of \Cref{lemma:Anproportionality_implies_Snproportionality} is false: characters of \(\hsss n\) labelled by \(\la\) and \(\al\) being proportional does not imply the characters of \(\haaa n\) labelled by \(\la\) and \(\al\) are proportional.
For example, let $\la=(2^2)$ and $\al=(4)$.
Then $\br{\symchar\la}=\br{\spsymchar{\al}}$ is the irreducible $2$-modular Brauer character of $\hsss4$ of degree $2$.
On restriction to $\haaa4$, this splits as the sum of two distinct Brauer characters of degree~\(1\), which arise as $\br{\altchar\la}$ and $\br{\conj{(\altchar\la)}}$, while $\br{[\al]}$ is the sum of these two irreducibles (so is not proportional to either of them).
\end{remark}

\begin{proof}[Proof of \Cref{thm:main}]
Suppose \(\br{\altchar{\la}}\) is proportional to
one (or both) of \(\br{\spaltchar{\al}}\) and \(\br{\conj{\spaltchar{\al}}}\).
Then \Cref{lemma:Anproportionality_implies_Snproportionality} implies that \(\br{\symchar{\la}}\) is proportional to \(\br{\spsymchar{\al}}\), and so by \cite[Theorem~1.1]{fayersmcd2025spintospecht} we have \(\la = \corandquot{\twoc{a}}{\twoc{r}}{\twoc{s}}\) and \(\al = \btwoc{a} \sqcup 2(\twoc{r}+\twoc{s})\) for some nonnegative integers \(a\), \(r\) and \(s\).
Note that \(\la\) is self-conjugate if and only if \(r=s\); we claim that if this is the case, then \(r=s=0\).
Indeed, if \(\la\) is self-conjugate,
then \(\br{\altchar{\la}}\) takes irrational values on \(g \in \haaa n\) projecting to cycle type \(\diag(\la)\) (necessarily consisting of distinct odd parts). 
Thus \(\br{\spaltchar{\al}}\) must also take irrational values on these elements, and so \(\al = \diag(\la)\).
But in order for \(\al = \btwoc{a} \sqcup 2(\twoc{r}+\twoc{s})\) to consist of distinct odd parts, we must have \(r=s=0\).

Conversely, suppose \(\la = \corandquot{\twoc{a}}{\twoc{r}}{\twoc{s}}\) and \(\al = \btwoc{a} \sqcup 2(\twoc{r}+\twoc{s})\) for some nonnegative integers \(a\), \(r\) and \(s\), with either \(r \neq s\) or \(r=s=0\).
By \cite[Theorem~1.1]{fayersmcd2025spintospecht} we have \( \br{\spsymchar{\al}} = 2^{\floor{\max\{r,s\}/2}} \br{\symchar{\la}} \).
Then by \Cref{lemma:iffrestrictiontoAn} we have \( \br{\spsymchar{\al}\downto{\haaa n}}= 2^{\floor{\max\{r,s\}/2}} \br{\symchar{\la}\downto{\haaa n}} \).

If \(r \neq s\), then \(\la\) is not self-conjugate, so \(\symchar{\la}\downto{\haaa n} = \altchar{\la}\) is irreducible.
If furthermore \(\max\{r,s\}\) is odd, then \(\al\) is odd, so \(\spsymchar{\al}\downto{\haaa n} = \spaltchar{\al} = \conj{\spaltchar{\al}}\) is also irreducible
and \( \br{\spaltcharsc{\al}} = \br{\conj{\spaltcharsc{\al}}} = 2^{\floor{\max\{r,s\}/2}} \br{\altcharnsc{\la}} \) as required.
If instead \(\max\{r,s\}\) is even,
then \(\al\) is even,
so \(\spsymchar{\al}\downto{\haaa n} = \spaltchar{\al} + \conj{\spaltchar{\al}}\),
but (since \(\al\) necessarily has an even part due to \(r\) and \(s\) being distinct) the only classes on which \(\spaltchar{\al}\), \(\conj{\spaltchar{\al}}\) differ are even order, so \(\br{\spaltchar{\al}} = \br{\conj{\spaltchar{\al}}} = \frac12 \br{\spsymchar{\al}\downto{\haaa n}}\); then \( \br{\spaltchar{\al}} = \br{\conj{\spaltchar{\al}}} = 2^{\floor{\max\{r,s\}/2} +1} \br{\altchar{\la}} \) as required.

If \(r=s=0\), then \(\la = \twoc{a}\) is self-conjugate and \(\al = \btwoc{a} = \diag(\la)\). 
Thus we have \(\br{\altchar{\la} + \conj{(\altchar{\la})}} = \br{\spaltchar{\al} + \conj{\spaltchar{\al}}}\).
But  \(\br{\altchar{\la}}\) and \(\br{\conj{(\altchar{\la})}}\) are irreducible \cite[Theorem~1.1]{benson1988Ansimplesmod2}, so \(\br{\spaltchar{\al}}\), \(\br{\conj{\spaltchar{\al}}}\) must be the same pair of irreducible characters in some order.
(Alternatively,
observe that all four of \(\br{\altchar{\la}}\), \(\br{\conj{(\altchar{\la})}}\), \(\br{\spaltchar{\al}} \) and \(\br{\conj{\spaltchar{\al}}}\)
agree on all classes except classes projecting to cycle type \(\al\),
and that on these classes each conjugate pair differs by the same quantity, namely \(\sqrt{ {(-1)}^{\frac{n-l}{2}}\al_1 \cdots \al_{l}}\) where \(l\) is the length of \(\al\) \cite[Theorem~2.5.13]{jameskerber1984reptheory} \cite[Theorem~8.7(iv)]{hoffmanhumphreys1992projreps}.)
\end{proof}

\section*{Acknowledgements}

\noindent The author thanks Matthew Fayers and John Murray for comments on an earlier version of this paper.

\bibliographystyle{alpha-noxc}
\bibliography{references}

\newcommand\ATlabel{}\newcommand\AT[2]{ATLAS}
\begin{thebibliography}{F{\relax McD}26}

\bibitem[\AT85]{ATLAS}
\ATlabel{\initials{J.H.} Conway, \initials{R.T.} Curtis, \initials{S.P.} Norton, \initials{R.A.} Parker and \initials{R.A.} Wilson}.
\newblock {\em {\(\mathbb{ATLAS}\)} of finite groups}.
\newblock Oxford University Press, Eynsham, 1985.
\newblock With computational assistance from J. G. Thackray.

\bibitem[Ben88]{benson1988Ansimplesmod2}
Dave Benson.
\newblock Spin modules for symmetric groups.
\newblock {\em Journal of the London Mathematical Society (Second Series)}, 38(2):250--262, 1988.

\bibitem[F{\relax McD}25]{fayersmcd2025spintospecht}
Matthew Fayers and Eoghan {\relax McD}owell.
\newblock Spin characters of the symmetric group which are proportional to linear characters in characteristic $2$.
\newblock {\em Annals of Representation Theory}, 2(1):37--83, 2025.

\bibitem[F{\relax McD}26]{fayersmcd2026proportionalspinpairs}
Matthew Fayers and Eoghan {\relax McD}owell.
\newblock Spin characters of symmetric and alternating groups which are proportional in characteristic $3$, 2026.
\newblock Preprint, \href{https://arxiv.org/abs/2606.10622}{arXiv:2606.10622}.

\bibitem[HH92]{hoffmanhumphreys1992projreps}
{\initials{P.N.}}~Hoffman and {\initials{J.F.}}~Humphreys.
\newblock {\em Projective Representations of the Symmetric Groups: Q-Functions and Shifted Tableaux}.
\newblock Oxford Mathematical Monographs. Clarendon Press, 1992.

\bibitem[JK84]{jameskerber1984reptheory}
Gordon James and Adalbert Kerber.
\newblock {\em The Representation Theory of the Symmetric Group}.
\newblock Encyclopedia of Mathematics and its Applications. Cambridge University Press, 1984.

\bibitem[McD24]{mcdowell2024charsonlprimeclasses}
Eoghan McDowell.
\newblock Characters and projective characters of alternating and symmetric groups determined by values on \(l'\)-classes.
\newblock {\em Arkiv f\"{o}r Matematik}, 62, 2024.

\bibitem[Wil08]{wildon2008distinctrows}
Mark Wildon.
\newblock Character values and decomposition matrices of symmetric groups.
\newblock {\em Journal of Algebra}, 319(8):3382--3397, 2008.

\end{thebibliography}

\end{document}